\documentclass[preprint]{elsarticle}
\usepackage{amsfonts, amssymb, amscd, graphicx, latexsym, amsmath, mathrsfs}

\usepackage[mathscr]{euscript}
\usepackage[latin5]{inputenc}
\usepackage[all]{xy}

\newcommand{\QQ}{\mathbb Q}

\newcommand{\ZZ}{\mathbb Z}

\newcommand{\squ}{\preceq}

\def\out{\rm Out}
\def\ind{{\rm Ind}}
\def\res{{\rm Res}}

\def\inf{{\rm Inf}}
\def\infl{{\rm Inf}}
\def\defl{{\rm Def}}
\def\jef{{\rm Jef}}
\def\jnf{{\rm Jnf}}
\def\ten{{\rm Ten}}

\def\iso{{\rm Iso}}
\def\lin{{\rm Lin}}
\def\hom{{\rm Hom}}

\newcommand{\rest}{{\rho_k(G)}}
\newcommand{\tran}{{\tau_k(G)}}
\newcommand{\conj}{{\gamma_k(G)}}
\newcommand{\mack}{{\mbox{\rm Mack}}}

\newtheorem{thm}{Theorem}
\newtheorem{pro}[thm]{Proposition}
\newtheorem{cor}[thm]{Corollary}
\newtheorem{lem}[thm]{Lemma}

\newtheorem{rem}[thm]{Remark}
\newtheorem{defn}[thm]{Definition}
\newtheorem{ex}{Example}

\newenvironment{proof}{\noindent{\bf Proof}\ }{$\square$}

\begin{document}

\begin{frontmatter}
\title{The Dade group of Mackey functors for $p$-groups}%

\author[label1]{Olcay Coşkun}

\address{Boğaziçi Üniversitesi, Matematik Bölümü, Bebek, İstanbul, Turkey }
\address{and}
\address{Feza Gürsey Center for Physics and Mathematics, Boğazici University, Kandilli 34684 İstanbul, Turkey}
\ead{olcay.coskun@boun.edu.tr}

\begin{abstract}
%% Text of abstract
We introduce endo-permutation Mackey functors and the Mackey-Dade group of a $p$-group and study their basic properties. 
We exhibit relations between the Mackey-Dade group of a finite $p$-group $P$ and the Dade groups of the Weyl groups
of all subgroups of $P$. As a result, we show that the Mackey-Dade group tensored with $\QQ$ over $\ZZ$ is the kernel of the linearization map
for Mackey functors, introduced by the author. 
\end{abstract}

\end{frontmatter}
\section{Introduction} 

In \cite{D1}, Dade introduced the family of endo-permutation $kP$-modules as an efficient way of studying 
representations of $p$-groups over a field of characteristic $p$. As shown by Dade, the class of endo-permutation modules
behaves nicely under the usual operations of taking duals, taking direct summands, restriction and inflation. He also 
constructed a deflation map on endo-permutation modules, now called Dade's slash construction.
A framework to study endo-permutation modules is the Dade group.

The structure of the Dade group, and hence the classification of all endo-permutation modules, is determined by Bouc in 
\cite{BDade}, where results of Alperin, Bouc, Bouc-Th\'{e}venaz, Bouc-Mazza and Carlson-Th\'{e}venaz are used. We refer
to \cite{B96} and references therein for a complete account of the classification.

After the classification problem is solved, there appear two kinds of generalizations of endo-permutation modules. In \cite{L} 
and in \cite{U}, there appear two attempts to generalize the notion of an endo-permutation module to an arbitrary finite group $G$. 
Whereas in \cite{LM}, the Dade group of a fusion system over a finite $p$-group is constructed.

In this paper, we introduce another of generalization of endo-permutation $kP$-modules by defining endo-permutation
Mackey functors. We show that it is possible to construct a group structure on the class of endo-permutation Mackey functors
for $P$, called the Mackey-Dade group of $P$. It turns out that the Mackey-Dade group shares some of the interesting 
properties of the Dade group. For example, we show that the torsion-free rank of the Mackey-Dade group of $P$ is the 
number of conjugacy classes of non-cyclic subquotients, whereas, by a result of Bouc and Th\'{e}venaz, \cite[Theorem A]{BT}, 
that of the Dade group is the number of conjugacy classes of non-cyclic subgroups of $P$. Note that the result for the Dade 
group also follows from \cite[Theorem D]{BT}, where it is shown that the Dade group tensored with $\QQ$ over $\ZZ$ is the kernel of 
the well-known linearization map from the Burnside ring to the ring of rational characters. We obtain the result on the rank of the 
Mackey-Dade group by constructing an isomorphism between this group and the product of all Dade groups of Weyl groups of all 
subgroups of $P$. Existence of such an isomorphism also suggests a relation between the Mackey-Dade group, the ring of 
subquotients of $P$, and the lineralization map for Mackey functors, both introduced by the author in \cite{Clin}. We explain this 
relation in Section \ref{sec:DvsL}.

It is shown in \cite{BT} that, after extending the coefficients to a field of characteristic zero, the Dade group functor becomes a
$p$-biset functor. Unfortunately, we do not have a version of this result for the Mackey-Dade group. The difficulty is with the
tensor induction map. 

The organization of the paper is as follows. In Section \ref{sec:prelim}, we collect necessary preliminaries on Mackey functors. 
Section \ref{sec:dual} introduces a new duality functor on the category of Mackey functors which is used to determine inverse 
elements in the Mackey-Dade group of a finite $p$-group. The definition and basic properties of endo-permutation Mackey functors 
are discussed in
Section \ref{sec:epmf} where we also introduce the Mackey-Dade group. The last two sections relate the Mackey-Dade group 
to known constructions. In Section \ref{sec:dd}, we show that the Mackey-Dade group of $P$ is a product of Dade groups of the
groups $N_P(Q)/Q$ as $Q$ runs over conjugacy classes of subgroups of $P$ and in Section \ref{sec:DvsL}, we show that,
after extending the coefficients to a field of characteristic zero, the Mackey-Dade group is isomorphic to the kernel of the
linearization map introduced in \cite{Clin}.

%%%%%%%%%%%%%%%%%%%%%%%PRELIMINARIES%%%%%%%%%%%%%%%%%%%%%%%%%%
\section{Preliminaries}\label{sec:prelim}
In this section, we collect some results on Mackey functors that will be used throughout the paper. We refer to 
\cite{BGreen}, \cite{BNon} and \cite{TW} for details and proofs. Throughout the paper, we assume that $k$ is a field and
all modules are finitely generated left modules over the given algebra, unless otherwise specified, and $G$ denotes a finite group. 
Note that although certain constructions and results are valid also over a commutative ring, our main results are only valid over fields. 
Thus we work over a field to simplify statements.

\subsection{Mackey functors}
Let $G$ be a finite group and $k$ be a field. Consider the free $k$-algebra on variables $c^g_H, r^H_K,
t_K^H$ where $K\le H\le G$ and $g\in G$. We define the \textit{Mackey algebra} $\mu_k(G)$ for $G$ over $k$ as the quotient
of this free algebra by the ideal generated by the following six relations:

\smallskip

\noindent {\bf (1)} $c^h_H = r_H^H = t_H^H$ for any $H\le G$ and any $h\in H$.

\smallskip

\noindent {\bf (2)} $ c^{g^\prime}_{{}^{g}H}c^{g}_H = c^{g^\prime g}_H \;\;\; {\rm and}\;\;\; r^K_L r^H_K = r^H_L
\;\;\; {\rm and} \;\;\; t_K^H t_L^K = t_L^H$ for any $L\le K\le H\le G$ and $g,g^\prime \in G$.

\smallskip

\noindent {\bf (3)}  $c^g_K r^H_K = r^{{}^gH}_{{}^gK} c^g_H \;\;\; {\rm and} \;\;\;
c^g_H t_K^H = t_{{}^gK}^{{}^gH} c^g_K$ for any $g\in G$ and $K\le H\le G$.

\smallskip

\noindent  {\bf (4)} $r^H_J t_K^H = \sum_{x\in J\backslash H/K}t_{J\cap {}^xK}^J c^x_{J^x\cap K} r^K_{J^x \cap K} \;\;\;
{\rm for} \;\;\; J\le H\ge K$  (Mackey Relation)

\smallskip

\noindent {\bf (5)} $\sum_{H\le G} r^H_H = 1$

\smallskip

\noindent {\bf (6)} All other products of generators are zero.

It is known that, letting $H$ and $K$ run over the subgroups of $G$ and letting $g$ run over the double coset
representatives $HgK \subset G$ and letting $L$ run over representatives of the subgroups of $H^g\cap K$ up to conjugacy,
the elements $t_{{}^gL}^H\, c^g_L \, r^K_L$ run (without repetitions) over the elements of an $k$-basis for the
Mackey algebra $\mu_k(G)$ (cf.\cite[Section 3]{TW2}).

We denote by $\rest$, called the \textit{restriction algebra} for $G$ over $k$, the subalgebra of the Mackey algebra
generated by $c^g_H$ and $r^H_K$ for $K\le H\le G$ and $g\in G$. We denote by $\tran$ the \textit{transfer algebra} for
$G$ over $k$, defined as the subalgebra generated by $c^g_H$ and $t^H_K$ for $K\le H\le G$ and $g\in G$.
The \textit{conjugation
algebra}, denoted $\conj$, is the subalgebra generated by the elements $c^g_H$. When there is no ambiguity, we write
$\mu = \mu_k(G)$, and $\rho = \rest$ and $\tau = \tran$ and $\gamma = \conj$. The restriction algebra $\rho$ has
generators $c^g_J \, r_J^K$, the transfer algebra $\tau$ has generators $ t_{^gJ}^Kc^g_J$ and the conjugation algebra
$\gamma$ has generators $c^g_J$.

We refer to \cite{BGreen}, \cite{C1}, \cite{Dr}, \cite{TW} and \cite{TW2} for further theoretical background on Mackey functors. 
In this paper, we use any of the following three definitions of a Mackey functor for $G$ over $k$: A Mackey functor $M$ for $G$ over $k$ is 
\begin{enumerate}
\item[(i)] a functor from the category of subgroups of $G$ to the category $k$-mod of finite-dimensional $k$-vector spaces,
\item[(ii)] a functor from the category of finite $G$-sets to $k$-mod,
\item[(iii)] a finitely generated left $\mu_k(G)$-module.
\end{enumerate} 

In \cite{TW}, Thévenaz and Webb constructed an anti-automorphism $(-)^*$ of the Mackey algebra $\mu_k(G)$ which allow us to 
define a right $\mu_k(G)$-module structure on Mackey functors. With this right module structure, we form, in the category 
Mack$(G)$ of Mackey functors, the tensor product functor $-\widehat\otimes_{\mu_k(G)} -$ and the hom functor 
Hom$_{\mbox{\rm Mack}(G)}(-, -)$. These functors satisfies the usual properties of the tensor product and the hom functor for 
module categories. We refer to \cite[Chapter 1]{BGreen} for further details. 

For our purposes, we need multiplicative functors between categories of Mackey functors preserving certain properties (e.g. being 
endo-permutation). It turns out that the usual 
functors of induction, inflation, and deflation defined in \cite{TW} are not suitable. A general theory of 
functors between categories of Mackey functors is described by Bouc in \cite{BGreen} in details, together with their left and 
right adjoints. Fortunately, two of the multiplicative functors that we are looking for are among these functors.
Below, we include the definition of the multiplicative inflation and the multiplicative deflation from \cite{BGreen} together
with the properties of these functors that are used in this paper. Note that these are not naturally equivalent to the usual
inflation and deflation functors introduced in \cite{TW}.

Following \cite{TW}, for $H\le G$ and $N\unlhd G$, we write $\res^G_H$ and $\infl_{G/N}^G$ for the ordinary restriction 
and 
inflation functors, respectively. We also have the following two functors. For $N\unlhd G$, define the multiplicative deflation 
functor (cf. \cite[Example 9.9.2]{BGreen})
\[
\jef^G_{G/N}: \mack(G)\rightarrow \mack(G/N)
\]
so that for a Mackey functor $M$ for $G$ and $N\le K\le G$, we have
\[
\jef^G_{G/N}(M)(K/N) = M(K)\Big/ \Big( \sum_{L}\mbox{\rm Im}(t_L^K: M(L)\rightarrow M(K)) \Big)
\] 
where $L$ runs over all subgroups of $K$ not containing $N$. The transfer, restriction and conjugation maps are the ones 
induced by those of $M$ on the corresponding quotients. 

In the reverse direction, we denote by
\[
\jnf^G_{G/N} :\mack(G/N)\rightarrow \mack(G)
\]
the functor $\iota_{G/N}^G$ described in \cite[Section 9.9.3]{BGreen}. In general, it is difficult to determine the evaluations of the Mackey functors $\jnf_{G/N}^G M$. However, by the 
remark at the end of Section 9.9 in \cite{BGreen}, we have
\begin{equation}\label{eqn:jnf1}
\jnf_{G/N}^G M(1) = M(N/N).
\end{equation}

Now let $G$ and $G^\prime$ be groups and $\lambda:G\rightarrow G^\prime$ be an isomorphism. Then we denote by
$\iso_{G,G^\prime}^\lambda$ the functor that transports Mackey functors for $G^\prime$ to Mackey functors for $G$ via
the isomorphism $\lambda$.

A final operation on Mackey functors that is used throughout the paper is the bar construction. Let $M$ be a Mackey functor
for $G$ over a field $k$ and $H\le G$. Following Th\'evenaz and Webb \cite[Section 2]{TW}, we define the $kN_G(H)/H$-
module $\overline M(H)$ by
\[
\overline M(H) = M(H)\Big/ \Big(\sum_{L< H} \mbox{\rm Im}\, t_L^H: M(L)\rightarrow M(H)\Big).
\]
Here the action of $N_G(H)/H$ is given by the conjugation maps $c^g_{N_G(H)/H}$ for $g\in N_G(H)/H$, (cf. \cite[Proposition 3.4]{TW}).
Via this definition, we associate a sequence $\overline M=(\overline M(H))_{H\le G}$ of modules to any Mackey functor $M$. 
It is clear that the sequence $\overline M$ is a $\gamma_k(G)$-module, and that $\overline M(1) = M(1)$ by definition.

On the other hand, by the remark at the end of Section 10.8 in \cite{BGreen}, we have 
\begin{equation}\label{eqn:barLU}
\overline M(H) = (\jef^{N_G(H)}_{N_G(H)/H}\res^G_{N_G(H)} M)(H/H).
\end{equation}
The following proposition describes the compositions of the above functors with the bar construction. Note that the proof
follows from the above isomorphism and the properties of Bouc's general construction $\mathcal L$ given in \cite{BGreen}.

\begin{pro}\mbox{\rm{(Bouc)}}\label{pro:barres} \label{pro:barjef}\label{pro:barjnf} 
Let $G$ be a finite group, $N\unlhd G$ and $H\le G$. Let $M$ be a Mackey functor for $G$ over $k$ and $M^\prime$ be a 
Mackey functor for $G/N$ over $k$.
\begin{enumerate}
\item Given $H\le K\le G$, there is an isomorphism
\[
\overline{\res^G_K M}(H) \cong \res^{N_G(H)/H}_{N_K(H)/H} (\overline M(H))
\]
of $kN_K(H)/H$-modules.
\item Suppose $N\le H\le G$. Then there is an isomorphism of $kN_G(H)/H$-modules
\[
\overline{\jef^{G}_{G/N} M} (H/N) \cong \overline M (H) 
\]
where we identify the groups $N_G(H)/H$ and $(N_G(H)/N)/(H/N)$ via the canonical quotient map.
\item There is an isomorphism
\[
\overline{\jnf_{G/N}^G M^\prime} (H) \cong \inf_{N_G(H)/N_{HN}(H)}^{N_G(H)/H}\big(\iso^\lambda(\overline{M^\prime}(HN/N))\big)
\] 
of $kN_G(H)/H$-modules where $\iso^\lambda$ denotes the transport of structure map induced by the canonical isomorphism
\[
\lambda: \frac{(N_G(H)N)/N}{(HN)/N} \rightarrow \frac{N_G(H)/H}{(N_G(H)\cap N)H/H}.
\]
\end{enumerate}
\end{pro}

%%%%%%%%%%%%%%%%%%%%%%%%%%%%%%%%%%%%%%%%%%%%%%%%%%%%%%%%%%%%
\section{Twin-dual construction}\label{sec:dual}
%%%%%%%%%%%%%%%%%%%%%%%%%%%%%%%%%%%%%%%%%%%%%%%%%%%%%%%%%%%%
In this section, we introduce the twin-dual of a Mackey functor $M$ on a finite group $G$, as a variation of the duality 
$M\mapsto \hom_k(M,k)$.

First recall the twin functor construction of Th\'evenaz \cite{T}. Given a Mackey functor $M$ for $G$, the \textit{twin functor} $TM$ of $M$ is the Mackey functor such that for any subgroup $Q\le G$, we have
\[
TM (Q) = \Big(\bigoplus_{K\le Q} \overline M(K) \Big)^Q
\]
where $(?)^Q$ denotes the submodule consisting of the $Q$-fixed points under the conjugation action of $Q$ coming from 
the Mackey functor structure. See Section 9 of \cite{TW2} for the action of the Mackey algebra. In this paper, we slightly modify 
this construction to define the twin-dual functor of a Mackey functor.
Given a $\gamma$-module $C$, write $C^*$ for its dual defined by $C^*(Q) = (C(Q))^* = \hom_k(C(Q),k)$ for any subgroup $Q$
of $G$. Then we define the \textit{twin-dual} $M^\circ$ of a Mackey functor $M$ to be the Mackey functor mapping $Q\le G$
to the $k$-module 
\[
M^\circ (Q) = \Big(\bigoplus_{K\le Q} (\overline M(K))^* \Big)_Q
\]
where $(X)_Q$ denotes the $Q$-coinvariants of $X$, i.e. the largest quotient of $X$ on which $Q$ acts trivially. The action of the Mackey algebra
is induced by its action on the Mackey functor $M$. Furthermore, by
Theorem 3.4.1 and Proposition 3.6 in \cite{C1}, we have
\begin{equation}\label{eqn:twin-dual}
\overline{M^\circ} (Q) = (\overline M(Q))^* \,\,\, \mbox{\rm for all} \,\,\, 1\le Q\le G.
\end{equation}
The following result clarifies the relation between the twin-dual and the dual of a Mackey functor. The proof of the result follows
easily from the duality theorems in \cite[Section 4]{C1} and by \cite[Theorem 12.2]{T}. Recall that, by our standing assumption,
any Mackey functor is a finite-dimensional $k$-vector space.
\begin{pro}
Let $M$ be a Mackey functor for $G$ over the field $k$. Then there is an isomorphism 
\[
M^\circ \cong (TM)^*
\]
of Mackey functors. In particular, if $k$ has characteristic not dividing the order of $G$, the twin-dual $M^\circ$ is isomorphic to 
the dual $M^*$ of $M$.
\end{pro}
%%%%%%%%%%%%%%%%%%%%%%%%%%%%%%%%%%%%%%%%%%%%%%%%%%%%%%%%%%%%%
\section{Endo-permutation Mackey functors}\label{sec:epmf}
%%%%%%%%%%%%%%%%%%%%%%%%%%%%%%%%%%%%%%%%%%%%%%%%%%%%%%%%%%%%%

For the rest of the paper, let $k$ be an algebraically closed field of characteristic $p>0$ and $P$ be a finite $p$-group. 
In this section, we introduce endo-permutation Mackey functors for finite $p$-groups and weakly-permutation Mackey functors for 
any finite group.

\begin{defn}
A Mackey functor $M$ for $G$ is called a \textit{\textbf{weakly-permutation Mackey functor}} if for any subgroup $H$ of $G$,
the $kN_G(H)/H$-module $\overline{M}(H)$ is a permutation module. In particular, $\overline{M}(1) = M(1)$ is a permutation
$kG$-module.
\end{defn}

In \cite[Definition 3.2]{BNon}, Bouc defines a permutation Mackey functor to be a Mackey functor isomorphic to $B_X$ for some
$G$-set $X$. One can show that any permutation Mackey functor is a weakly-permutation Mackey 
functor, for example by \cite[Lemma 5.10]{Bproj}. However, in general, there are weakly-permutation Mackey functors which are
not of the form $B_X$. Indeed, any permutation Mackey functor is projective and hence it is sufficient to exhibit a non-projective
weakly-permutation Mackey functor. As an example, consider the Mackey functor
\[
M:= \bigoplus_{H\le_G G} S_{H,1}^\mu
\]
where the sum is over conjugacy classes of subgroups of $G$ and $S_{H,1}^\mu$ denotes the simple Mackey functor with minimal 
group $H$ and $S_{H,1}^\mu(H) \cong k$. It can be shown that for any subgroup $H\le G$, we have $\overline{M}(H) \cong k$ as 
$kN_G(H)/H$-modules, but $M$
is not projective unless the characteristic of $k$ does not divide the order $|N_G(H)/H|$ for each $H$ in $G$. Therefore, when
$k$ has characteristic $p$, and $G$ has order divisible by $p$, the functor $M$ is an example of a
non-projective weakly-permutation Mackey functor which is not a permutation Mackey functor.

We are now ready to introduce the main concept of this paper.
\begin{defn}
Let $P$ be a finite $p$-group and $k$ be a field of characteristic $p$. 
A Mackey functor $M$ for $P$ is called an \textbf{endo-permutation Mackey functor} if 
the Mackey functor $M^\circ\widehat\otimes M$ for $P$ is a weakly-permutation Mackey functor. 
\end{defn}

Note that, by \cite[Proposition 1.9.1]{BGreen}, there is an isomorphism of Mackey functors $M^\circ\widehat\otimes M \cong 
M^\circ\widehat\otimes M$. Thus, equivalent to the above definition, a Mackey functor $M$ is an endo-permutation  Mackey functor
if $M^\circ\widehat\otimes M$ is a weakly-permutation Mackey functor.

It is possible to characterize endo-permutation Mackey functors in terms of endo-permutation
$kP$-modules, introduced by Dade in \cite{D1}. Given a finite $p$-group $P$ and a field $k$ of characteristic $p$, a $kP$-
module $M$ is called an \textit{endo-permutation} $kP$-module if End$_k(M)$ is a permutation $kP$-module. Here End$_k(M)$ is
a $kP$-module via the action given by
\[
(g\cdot \phi) (m) = g\cdot \phi(g^{-1}\cdot m)
\]
for all $g\in G, \phi \in \mbox{\rm End}_k(M)$ and $m\in M$. Note that, putting $M^* = \hom_k(M,k)$, we have 
End$_k(M) = \hom_k(M,M) \cong M^*\otimes M$.
\begin{lem}
Let $M$ be a Mackey functor for $P$ over $k$. Then the following statements are equivalent.
\begin{enumerate}
\item $M$ is an endo-permutation Mackey functor.
\item For each subgroup $Q$ of $P$, the $kN_P(Q)/Q$-module $\overline M(Q)$ is an endo-permutation module.
\end{enumerate}
\end{lem}
\begin{proof}
By \cite[Lemma 10.8.5]{BGreen}, we have
\[
\overline{M^\circ\widehat\otimes M}(Q) \cong \overline{M} (Q)\otimes \overline{M^\circ}(Q)
\]
where the tensor product on the right hand side is the usual tensor product over $k$ of $kN_P(Q)/Q$-modules. 
Also by the isomorphism given in (\ref{eqn:twin-dual}), we have $\overline{M^\circ}(Q) = (\overline{M}(Q))^*$. Therefore, 
we get 
\[
\overline{M^\circ\widehat\otimes M}(Q) \cong (\overline{M}(Q))^*\otimes \overline{M}(Q).
\]
Now the result follows from the definition of an endo-permutation module.
\end{proof}

As remarked in \cite[p. 1874]{TW}, the Mackey algebra is a finite dimensional $k$-algebra and hence Krull-Schmidt Theorem
applies and we have projecitve modules, simple modules, indecomposable modules etc. Also, by \cite{S} (cf. \cite[Section 11]{TW}), there
is a theory of vertices and sources for projective and simple Mackey functors. Now these remarks and the above result 
allow us to define a capped endo-permutation Mackey functor.

\begin{defn}
An endo-permutation Mackey functor $M$ for $P$ is called \textit{\textbf{capped}} if for each subgroup $Q\le P$, the 
endo-permutation $kN_P(Q)/Q$-module $\overline M(Q)$ is capped.
\end{defn}

Before looking at the properties of endo-permutation Mackey functors, we include some examples.

\begin{ex}
Any weakly-permutation Mackey functor $M$ for $P$ over $k$ is also an endo-permutation Mackey functor. Indeed, if $M$ is 
weakly-permutation, then the twin-dual functor $M^\circ$ is also weakly-permutation by Equation (\ref{eqn:twin-dual}), since the 
dual of a permutation module is also a permutation module. Hence $M^\circ\widehat\otimes M$ is  also a weakly-permutation Mackey 
functor. Therefore $M$ is an endo-permutation Mackey functor. 
\end{ex}

\begin{ex}
An example of a capped endo-permutation Mackey functor is the Burnside functor $kB^P$. Given a subgroup $Q\le P$,
the evaluation $kB^P(Q)$ of the Burnside functor at $Q$ is the $k$-linear extension $k\otimes B(Q)$ of the Burnside group $B(Q)$
of $Q$. Here the Burnside group $B(Q)$ is the Grothendieck group of the category of $Q$-sets with disjoint union as the 
addition. The action of the Mackey algebra is induced by the usual functors between categories of sets with a group action. For 
example, if $R\le Q\le P$, then $r^Q_R$ act as the restriction map from the category of $Q$-sets to the category of $R$-sets.

Now, by \cite[Proposition 6.2]{T}, we have $\overline{kB^P}(Q) \cong k$ for each $Q\le P$. Also, by \cite[Corollary 8.9]{TW}, the 
Burnside 
functor 
$kB^P$ is indecomposable, isomorphic to the projective cover $\mathcal P_{P,1}$ of the simple functor $S_{P,1}$. Thus the 
Burnside functor $kB^P$ is an indecomposable capped endo-permutation Mackey functor.
\end{ex}
For the rest of the paper, we only consider capped endo-permutation Mackey functors and hence capped endo-permutation
modules. Therefore, for the rest of the paper, we assume that all endo-permutation Mackey functors and all endo-permutation 
modules are capped. With this assumption, since the zero module is not a capped endo-permutation $kP$-module, for any finite $p$-group $P$,
the following holds.

\begin{cor}
Let $M$ be a capped endo-permutation Mackey functor for $P$ over $k$. Then $\overline M(Q)$ is non-zero for each 
subgroup $Q\le P$.
\end{cor}

Next we look at some properties of capped endo-permutation Mackey functors. Let $M$ and $N$ be capped endo-permutation 
Mackey functors for $P$ and $M^\prime$ be a direct summand of $M$. Then, in general, $M^\prime$ is not a capped 
endo-permutation Mackey functor since we may have $\overline{M^\prime}(Q) = 0$ for some subgroup $Q\le P$. However, since
any direct summand of an endo-permutation module is again endo-permutation, if $\overline{M^\prime} (Q)$ is capped for each 
subgroup $Q\le P$, then $M^\prime$ is a capped endo-permutation Mackey functor. For simplicity, a direct summand $M^\prime$ 
of $M$ such that $\overline{M^\prime} (Q)\neq 0$ for each subgroup $Q\le P$ is called \textit{\textbf{primordial}}. Thus, any primordial 
summand of a capped endo-permutation Mackey functor is a capped ando-permutation Mackey functor. 

On the other hand, since the bar construction functor $M\mapsto \overline M$ is additive, the direct sum $M\oplus N$ is an 
endo-permutation Mackey functor provided that for each $Q\le P$, the endo-permutation $kN_P(Q)/Q$-modules 
$\overline M (Q)$ and $\overline N(Q)$ are  compatible,  that is, $\overline M (Q)^*\otimes\overline N(Q)$ is a permutation 
$kN_P(Q)/Q$-module, \cite[Proposition 2.3]{D1}. If this condition holds, then we say that $M$ and $N$ are 
\textit{\textbf{compatible}} endo-permutation Mackey functors.

Moreover \cite[Lemma 10.8.5]{BGreen} yields
\[
\overline{M\widehat\otimes N}(Q) \cong \overline{M} (Q)\otimes \overline{N}(Q)
\]
for all Mackey functors $M$ and $N$ for $P$ and for any subgroup $Q$ of $P$. Here the tensor product on the right hand side is the usual tensor product over $k$ of $kN_P(Q)/Q$-modules. Thus, since the
tensor product of endo-permutation modules is again an endo-permutation module, by \cite[Proposition 2.2]{D1}, the Mackey 
functor $M\widehat\otimes N$ is also an endo-permutation Mackey functor. 

Note that, in general, the class of capped endo-permutation Mackey functors is not closed under taking duals, since 
$\overline{M^*}(Q)$ may be zero even though $\overline M(Q)$ is not. However, taking twin-duals work. Indeed, if $M$ is a 
Mackey functor, then by Equation (\ref{eqn:twin-dual}), we have 
$\overline{M^\circ}(Q) = (\overline{M}(Q))^*$. Further if $M$ is a capped endo-permutation Mackey functor, then the $kN_P(Q)/Q$-module 
$\overline{M}(Q)$ is a capped endo-permutation module. Also, by \cite[Proposition 2.2]{D1}, the dual module $(\overline{M}(Q))^*$ is a capped
endo-permutation module. Therefore, the twin-dual $M^\circ$ is a capped endo-permutation Mackey functor.

Finally, by Proposition \ref{pro:barres} and by \cite[Corollary 2.13]{BT}, the class of endo-permutation Mackey functors is 
preserved under the functors of restriction, multiplicative deflation, multiplicative inflation and the transport of structure.

We summarize these properties with the following theorem.
\begin{thm}
The class of capped endo-permutation Mackey functors is closed under the operations of tensor products, taking primordial 
direct summands, taking twin-duals and the under the functors  of restriction, multiplicative deflation, 
multiplicative inflation and the transport of structure introduced in the previous section. Moreover, direct sum of two compatible endo-permutation 
Mackey functors is also endo-permutation. 
\end{thm}

Our next aim is to construct the group of capped endo-permutation Mackey functors. Following Dade, we want to consider an equivalence 
relation on the isomorphism classes of capped endo-permutation Mackey functors. To have a comparison, recall that two capped  
endo-permutation $kP$-modules $m$
and $n$ are called \textbf{\textit{equivalent}}, written $m\sim n$ if they have isomorphic caps. To be more precise, note that the cap 
of a capped endo-permutation $kP$-module $m$ is the, unique up to isomorphism, indecomposable summand of $m$ with vertex $P$.
Thus, we have $m\sim n$ if $m$ and $n$ have a common indecomposable direct summand with vertex $P$.

Now in the case of Mackey functors, by \cite{S}, any indecomposable Mackey functor $M$ has a 
vertex. If $M$ is an endo-permutation Mackey functor, then it has an indecomposable 
summand with vertex $P$. Indeed, since, by definition, we have $\overline M (P) \neq 0$, there is an element $x\in M(P)$ 
such that the image $\bar x$ of $x$ in the quotient $\overline M(P)$ is not zero. Thus the subfunctor $\langle x\rangle$ of $M$ generated by 
the element $x$ has vertex $P$ since, by its choice, $x$ cannot be contained in $(\ind_Q^P\res^P_Q \langle x\rangle)(P)$.
However this summand is not necessarily primordial. Therefore we cannot define an equivalence relation on these functors 
using this summand with vertex $P$.
Instead, we consider the following relation.
\begin{defn}
Let $M$ and $N$ be capped endo-permutation Mackey functors for $P$. We say that $M$ and $N$ are \textbf{equivalent}
if for any subgroup $Q\le P$, the capped endo-permutation $kN_P(Q)/Q$-modules $\overline M(Q)$ and $\overline N(Q)$
are equivalent, with respect to Dade's equivalence of endo-permutation modules.
\end{defn}
It is straightforward to show that the above relation is an equivalence relation. Given a capped endo-permutation Mackey functor 
$M$, we denote by $[M]$ its equivalence class with respect to this equivalence relation. Moreover we denote the set of all
equivalence classes of endo-permutation Mackey functors for $P$ by $$D_\mu(P).$$
This set becomes an abelian group under the operation
\[
[M] + [N] = [M\widehat\otimes N] = [N\widehat\otimes M]
\]
where $M$ and $N$ are endo-permutation Mackey functors for $P$. Here, we use \cite[Proposition 1.9.1]{BGreen} for the
commutativity of the tensor product. Note that the identity element with respect to this operation is the class $[kB^P]$
of the Burnside functor since $$kB^P\widehat\otimes M \cong M\cong M\widehat\otimes kB^P$$ for any Mackey functor $M$ for $P$,
by Proposition 2.4.5 in \cite{BGreen}. Remark that the class $[kB^P]$ contains all endo-permutation Mackey functors 
$N$ such that $[\overline N (Q)] = [k]$ in the Dade group $D(N_P(Q)/Q)$ of $N_P(Q)/Q$. 

Regarding the inverses, as remarked earlier, the class of endo-permutation Mackey functors is not closed under taking duals
but it is closed under taking twin-duals. We claim that the inverse of the class $[M]$ is the class containing $M^\circ$.
Indeed, if $M$ is a capped endo-permutation Mackey functor for $P$, then $M^\circ$ is also a capped endo-permutation 
Mackey functor. Moreover by Equation (\ref{eqn:twin-dual}), we have
\[
\overline{(M^\circ\widehat\otimes M)}(Q) = \overline{M} (Q)\otimes \overline{M^\circ}(Q) = \overline{M} (Q)\otimes 
(\overline{M}(Q))^*.
\]
Therefore, in $D_\mu(P)$, we have
\[
[M^\circ\widehat\otimes M] = [kB^P].
\]
Thus we have proved the following theorem.
\begin{thm}
Let $P$ be a $p$-group and $k$ be an algebraically closed field of characteristic $p$. Then the set $D_\mu(P)$ of the 
equivalence classes of endo-permutation Mackey functors for $P$ over $k$ is an abelian group where
\begin{enumerate}
\item[(i)] the group operation is given by the tensor product of Mackey functors,
\item[(ii)] the identity element of the group is the class $[kB^P]$ of the Burnside functor, and
\item[(iii)] given an endo-permutation Mackey functor $M$, the inverse $-[M]$ of the class $[M]$ is the class containing the 
twin-dual  $M^\circ$ of $M$.
\end{enumerate}
The set $D_\mu(P)$ together with this group operation is called the \textbf{Dade group} of Mackey functors for $P$,
or, shortly, the \textbf{Mackey-Dade group} of $P$.
\end{thm}

\begin{rem}
Let $R\le P$. Then the restriction functor on Mackey functors induces  a group homomorphism
\[
\res^P_R: D_\mu(P) \to D_\mu(R).
\]
Similarly, If $N\unlhd P$ and $Q=P/N$, the multiplicative deflation and multiplicative inflation maps of the previous section induce group
homomorphisms
\[
\jef^P_Q: D_\mu(P)\to D_\mu(Q)
\]
and
\[
\jnf^P_Q: D_\mu(Q)\to D_\mu(P).
\]
Also, if $P$ and $P^\prime$ are isomorphic groups, then any isomorphism between them induces a transport of structure map
between the corresponding Mackey-Dade groups. However, we do not know a way to construct a multiplicative induction functor for 
endo-permutation Mackey functors, and therefore, at this point, we cannot equip $D_\mu$ with a structure of a biset functor. 
\end{rem}
%%%%%%%%%%%%%%%%%%%%%%%%%%%%%%%%%%%%%%%%%%%%%%%%%%%%%%%%%%%%%%
\section{Relating the Mackey-Dade group to the Dade group}\label{sec:dd}
%%%%%%%%%%%%%%%%%%%%%%%%%%%%%%%%%%%%%%%%%%%%%%%%%%%%%%%%%%%%%%
The Mackey-Dade group of $P$ is closely related to the Dade groups of $P$ and the Weyl groups of its subgroups, as the 
definition suggests. In this section, we clarify this relation by proving the following theorem.  
\begin{thm}\label{thm:dadeiso}
Let $P$ and $k$ be as above. Then there is an isomorphism of abelian groups
\[
D_\mu(P) \cong \bigoplus_{Q\le_P P} D(N_P(Q)/Q)
\]
where the sum is over a set of representatives of conjugacy classes of subgroups of $P$. Moreover a 
canonical isomorphism is given by associating an endo-permutation Mackey functor $M$ to the sequence 
$\big([\overline M(Q)]\big)_{Q\le_P P}$.
\end{thm}
\begin{proof}
Define
\[
\bar ? : D_\mu(P)\rightarrow \bigoplus_{Q\le_P P} D(N_P(Q)/Q)
\]
by $\overline{[M]} = \big([\overline M(Q)]\big)_{Q\le_P P}$. Clearly, this is an abelian group homomorphism. To show that it is also 
bijective, we construct an inverse. Let $(m_Q)_{Q\le_P P}$ be a sequence in the direct sum. For each subgroup $Q\le P$, denote
by $n_Q$ an endo-permutation $kN_P(Q)/Q$-module representing the class $m_Q$. Also let
$C_{Q,n_Q}$ be the atomic conjugation functor concentrated on the conjugacy class of $Q$ with $C_{Q,n_Q}(Q) = n_Q$. Finally,
let 
\[
Y_{Q,n_Q} = \ind_\rho^\mu\infl_\gamma^\rho C_{Q,n_Q}.
\]
With this notation, define the map
\[
\varrho_P : \bigoplus_{Q\le_P P} D(N_P(Q)/Q) \rightarrow D_\mu(P)
\]
by 
\[
\varrho_P((m_Q)_{Q\le_P P}) = \Big[\bigoplus_{Q\le_P P} Y_{Q,n_Q}\Big].
\]
For simplicity, we write $Y$ for the Mackey functor $\bigoplus_{Q\le_P P} Y_{Q,n_Q}$.
Now we need to check that $Y$ is a capped endo-permutation module, that is, we need to justify that the 
$kN_P(R)/R$-module $\overline{Y}(R)$ is a capped endo-permutation module for each subgroup $R\le P$. But by 
\cite[Theorem 3.4.1 and Proposition 3.6]{C1}, we have  $\overline{Y_{Q, n_Q}}(R) = 0$ unless $R$ is $P$-conjugate to $Q$ and 
\[
\overline{Y_{R, n_R}}(R) = n_R
\]
which is a capped endo-permutation $kN_P(R)/R$-module by its choice. Thus $Y$ is a capped endo-permutation Mackey functor.

On the other hand, to prove that the map $\varrho_P$ is well-defined, let $n_Q^\prime$ be another representative of the class $m_Q$
 for each $Q\le P$. Then since $\overline{Y_{Q, n_Q}}(Q) = n_Q$, for each $Q\le P$, the equalities
\[
\Big[\overline{\bigoplus_{Q\le_P P}Y_{Q,n_Q}}(R)\Big] = [n_R] =m_R = [n_R^\prime] =\Big[\overline{\bigoplus_{Q\le_P P}
Y_{Q,n_Q^\prime}}(R)\Big].
\]
hold in the Dade group of $N_P(R)/R$. Therefore, by definition, the equality
\[
\Big[\bigoplus_{Q\le_P P} Y_{Q,n_Q}\Big] = \Big[\bigoplus_{Q\le_P P} Y_{Q,n_Q^\prime}\Big]
\]
holds in the Mackey-Dade group of $P$, that is, the map is well-defined. The above proof also shows that
\[
\overline{\varrho_P((m_Q)_{Q\le_P P})} = ([\overline{Y}(Q)])_{Q\le_P P} = ([n_Q])_{Q\le_P P} = (m_Q)_{Q\le_P P}.
\]
Hence, we have proved that the composition $\overline{\varrho_P}$ is the identity endomorphism of $\oplus_Q D(N_P(Q)/Q)$.
On the other hand, if $M$ is an endo-permutation Mackey functor, we have
\[
\overline{\varrho_P([\overline M])} (R) = \overline{\bigoplus_{Q\le_P P} Y_{Q,\overline M(Q)}} (R) = \overline{M}(R)
\]
for each subgroup $R\le P$. Therefore $\varrho_P( \overline ?) = \mbox{\rm id}_{D_\mu(P)}$, as required.
\end{proof}

As a corollary to the above theorem, we determine the torsion-free rank of $D_\mu(P)$ as follows. For simplicity, we put 
$\mathbb Q D_\mu(P) = \mathbb Q\otimes_\mathbb Z D_\mu(P)$.
\begin{cor}
With the above notation, the $\mathbb Q$-dimension of $\mathbb Q D_\mu(P)$ is the number of conjugacy classes of non-cyclic
subquotients of $P$, that is,
\[
\mbox{\rm dim}_{\mathbb Q} \mathbb Q D_\mu(P) = |\{ (R,Q)| Q\unlhd R\le P: R/Q \mbox{ non-cyclic}\}/P|.
\]
\end{cor}
\begin{proof}
By the above theorem, we have
\[
\mbox{\rm dim}_{\mathbb Q} \mathbb Q D_\mu(P) = \sum_{Q\le_P P} \dim_\mathbb Q \mathbb Q D(N_P(Q)/Q).
\]
On the other hand, by \cite[Theorem A]{BT}, the torsion-free rank of $D(N_P(Q)/Q)$ is equal to the number $nc(N_P(Q)/Q)$ of
conjugacy classes of non-cyclic subgroups of $N_P(Q)/Q$. Thus we get 
\[
\mbox{\rm dim}_{\mathbb Q} \mathbb Q D_\mu(P) = \sum_{Q\le_P P} nc(N_P(Q)/Q).
\]
But note that if $R/N$ is a non-cyclic subquotient of $P$, then it is a subgroup of exactly one group of the form $N_P(Q)/Q$, 
namely of the one for which $Q$ and $N$ are conjugate. Therefore, as $Q$ runs over a complete set of representatives of 
conjugacy classes of $P$ and as $R/Q$ runs over a complete set of representatives of conjugacy classes of non-cyclic 
subgroups of $N_P(Q)/Q$, the groups $R/Q$ run over a complete set of representatives of conjugacy classes of non-cyclic 
subquotients of $P$. This completes the proof.
\end{proof}

A similar result can be obtained for the torsion-part of the Mackey-Dade group. On the other hand,
in the reverse direction, we can identify the Dade group $D(P)$ with a subgroup of the Mackey-Dade group $D_\mu(P)$ in a 
canonical way by identifying the Dade group in the above direct sum, as follows. 
\begin{cor}
The above isomorphism $\varrho_P$ induces a group monomorphism
\[
\varrho_P: D(P)\rightarrow D_\mu(P)
\]
given by associating an endo-permutation $kP$-module $m$ to the Mackey functor $$\bigoplus_{Q\le_P P} Y_{Q,m_Q}$$
where $m_Q = \defl^{N_P(Q)}_{N_P(Q)/Q}\res^P_{N_P(Q)}\, m$ for each $Q\le P$. 
\end{cor}

As another corollary, we have the following canonical representative of the equivalence class of an endo-permutation Mackey
functor.
\begin{cor}
Let $M$ be an endo-permutation Mackey functor for $P$.
Then the class of $M$ in $D_\mu(P)$ contains the Mackey functor
$$YM:=\bigoplus_{Q\le_P P} Y_{Q,\overline M(Q)}.$$ 
\end{cor}

The above corollary suggests another characterization of the equivalence classes of capped endo-permutation Mackey functors. Let 
$M$ be a capped endo-permutation Mackey functor. For each $Q\le P$, let $C(M, Q)$ denote the cap of the capped
endo-permutation $kN_P(Q)/Q$-module $\overline M(Q)$. 

\begin{defn}
Let $M$ be a capped endo-permutation Mackey functor for $P$ over $k$. We define the \textbf{\textit cap} of $M$ as the capped
endo-permutation Mackey functor
\[
\mbox{\rm cap} (M) = \bigoplus_{Q\le_P P} Y_{Q, C(M,Q)}
\]
for $P$ over $k$.
\end{defn}

Note that by \cite[p. 470]{D1}, the cap of a capped endo-permutation $kP$-module $M$ is the unique, up to isomorphism, indecomposable 
summand of $M$ with vertex $P$. As we have explained above, a capped endo-permutation Mackey functor $M$ for $P$ may not 
have an indecomposable summand with vertex $P$ which is a capped endo-permutation Mackey functor. With the above definition, the cap of $M$ is the 
unique, up to isomorphism, primordial summand cap$(M)$ of $M$ such that the vertex of the $kN_P(Q)/Q$-module $\overline{\mbox{\rm cap}(M)}(Q)$ is $N_P(Q)/Q$ for 
each subgroup $Q\le P$. 

Now the following corollary is immediate by the definition of Dade's equivalence of endo-permutation modules and our
definition of equivalence of endo-permutation Mackey functors.
\begin{cor}
Let $M$ and $N$ be endo-permutation Mackey functors. Then $M$ is equivalent to $N$ if and only if there is an isomorphism 
\[
\mbox{\rm cap} (M)\cong \mbox{\rm cap} (N)
\]
of Mackey functors for $P$.
\end{cor}
%%%%%%%%%%%%%%%%%%%%%%%%%%%%%%%%%%%%%%%%%%%%%%
\section{Ring of subquotients and the Mackey-Dade group}\label{sec:DvsL}
%%%%%%%%%%%%%%%%%%%%%%%%%%%%%%%%%%%%%%%%%%%%%%

In this section, we show that the Mackey-Dade group $\mathbb QD_\mu(P)$ tensored with $\mathbb Q$ over $\mathbb Z$ can be realized as the kernel of the 
linearization map defined by the author in \cite{Clin}. First we recall the definitions that we need. 

Let $P$ be a finite $p$-group. We write $\mathcal{SQ}(P)$ for the set of all subquotients of $P$, that is,
\[
\mathcal{SQ}(P) = \{(Q,N)| N\unlhd Q\le P\}
\]
and write $\mathcal{SQ}_P(P)$ for the set of $P$-orbits of the set $\mathcal {SQ}$. We sometimes write the pair $(Q,N)$ as 
$Q/N$, and $(Q,N)\in \mathcal{SQ}(P)$ as $Q/N\squ P$. Also, the $P$-orbit containing $(Q,N)$ is denoted by $[Q,N]_P$ and we
write $Q/N\squ_P P$ instead of $[Q,N]_P\in\mathcal{SQ}_P(P)$.

Now we denote by $\Lambda(P)$ the free abelian group on the set $\mathcal{SQ}_P(P)$, that is,
\[
\Lambda(P) = \bigoplus_{Q/N \squ_P P} \mathbb Z [Q,N]_P.
\]
This group becomes a unital commutative associative ring with the product given by the linear extension of the product defined on 
the basis elements by 
\[
[R,N]_P \cdot [S,M]_P = \sum_{\substack{x\in R\backslash P/S\\ {}^xMN\le R\cap{}^xS}} [R\cap {}^xS, {}^xMN].
\]
The unity of this ring is $[P/1]_P$. We call $\Lambda(P)$ the \textbf{\textit{ring of subquotients}} of $P$. Note that the Burside ring 
is a subring of $\Lambda(P)$ via the inclusion 
\[
[P/R]\mapsto [R,1]_P.
\]
We refer to \cite{Clin} for further details on the structure of $\Lambda(P)$. 

The reason for introducing the above product on subquotients is to define a linearization map for Mackey functors. By \cite{TW},
the Mackey algebra for $P$ over a field $\mathbb K$ of characteristic zero is semi-simple. We denote by  
$\mathcal R_{\mu_\mathbb K} (P)$ the Grothendieck ring of the category of 
Mackey functors for $P$ over $\mathbb K$. Here the addition and the multiplication operations are 
given, respectively, by the direct sum and the tensor product of Mackey functors. The set of isomorphism 
classes of simple Mackey functors forms a basis for the ring $\mathcal R_{\mu_\mathbb K}(P)$ and the class of the Burnside 
ring functor $\mathbb KB^P$ is the unit element. Therefore, by the classification of simple Mackey functors given in \cite{TW2}, there 
is an isomorphism
\[
\mathcal R_{\mu_\mathbb K} (P) \cong \bigoplus_{\substack{(H,V)\\ H\le_P, V\in \mbox{\rm{\small{Irr}}}(\mathbb KN_P(H)/H)}} \mathbb Z
[S_{H,V}^P]
\]
of abelian groups where $[S_{H,V}^P]$ denotes the isomorphism class of the simple Mackey functor $S_{H,V}^P$ for $P$ with 
minimal subgroup $H$ and $S_{H,V}^P(H) = V$.

Now the linearization map 
\[
\lin_P^\mu: \Lambda(P)\rightarrow \mathcal R_{\mu_\mathbb K} (P)
\]
is defined as the linear extension of the map associating the basis element $[Q,N]_P$ to the class of the Mackey functor
$\ind_Q^P\infl_{Q/N}^Q \mathbb K B^{Q/N}$. Here the inflation and induction maps are the usual ones as defined in \cite{TW}, or as 
in \cite{BGreen}. 

By \cite[Theorem 5.3]{Clin}, this linearization map is a ring homomorphism. Also by \cite[Theorem 5.5]{Clin}, the extension
\[
\mathbb{Q}\lin_P^\mu: \mathbb{Q}\Lambda(P)\rightarrow \mathbb{Q}\mathcal R_{\mu_\mathbb Q} (P)
\]
is surjective. Hence, we have an analogy with the well-known linearization map
\[
\lin_P: B(P)\rightarrow \mathcal R_{\mathbb K}(P)
\]
where $\mathcal R_\mathbb K(P)$ is the $\mathbb K$-character ring of $P$ and the map is given by associating the basis element 
$[P/Q]$ of the Burnside ring $B^P(P)$ to the permutation character $\ind_Q^P \mathbb K = \ind_Q^P\,\infl_{Q/Q}^Q \mathbb K$ 
where $\mathbb K$ is the trivial $\mathbb KG$-module for any finite group $G$. 

To push the analogy further, recall that by \cite[Theorem D]{BT}, the kernel of $\mathbb Q\lin_P$ is the torsion-free group 
$\mathbb QD(P)= \mathbb K\otimes_{\mathbb Z} D(P)$. Next, we show that this is also the case for Mackey functors. However, as we shall see in the next 
section, this identification is functorial only in a restricted sense. 

Henceforth, let $P$ be a finite $p$-group and let $s(P)$ denote a set of representatives of the conjugacy classes of subgroups of
$P$. To prove the analogue of \cite[Proposition 5.1]{BT} in terms of Mackey functors, we need the following result.

\begin{thm}\mbox{\rm \cite[Proposition 6.1]{TW}}\label{thm:tw18}
Let $P$ be a $p$-group and let $s(P)$ be as above. Then the bar construction induces an isomorphism of $\mathbb Q$-vector 
spaces
\[
\overline{?}: \mathbb Q\mathcal R_{\mu_\mathbb K}(P) \rightarrow \bigoplus_{R\in s(P)}\mathbb Q\mathcal R_\mathbb K
(N_P(R)/R)
\]
which maps a Mackey functor $M$ to the sequence $(\overline{M}(R))_{R\in s(P)}$. The inverse of $\overline ?$ maps a sequence 
$(m_R)_{R\in s(P)}$ to the Mackey functor $\bigoplus_{R\in s(P)} Y_{R,m_R}.$
\end{thm}
A similar result holds for the space $\mathbb Q\Lambda(P)$ of subquotients. Next we describe this result.

\begin{thm} Assume the above notation. Then the linear map
\[
\alpha_P: \mathbb Q\Lambda(P) \rightarrow \bigoplus_{R\in s(P)}\mathbb Q B(N_P(R)/R)
\]
extending the function given by mapping a basis element $[Q,N]_P$ of $\mathbb Q\Lambda(P)$ to the sequence $(a_R(Q,N))_{R\in s(P)}$ where
\[
a_R(Q,N) =  \sum_{\substack{g\in Q\backslash P/N_P(R):\\ N\le {}^gR\le Q}} [N_P(R)/N_{{}^{g^{-1}}Q}(R)] 
\]
is an isomorphism of $\mathbb Q$-vector spaces.
\end{thm}

\begin{proof}
First, we check that the map $\alpha_P$ is well-defined. Let $(T,S) = {}^x(Q,N)$ for some $x\in P$. Then
\begin{eqnarray*}
a_R(T,S) = a_R({}^xQ,{}^xN) &=& \sum_{\substack{h\in {}^xQ\backslash P/N_P(R):\\ {}^xN\le {}^hR\le {}^xQ}} [N_P(R)/N_{{}^{h^{-1}x}Q}(R)]\\
&=& \sum_{\substack{x^{-1}h\in Q\backslash P/N_P(R):\\ N\le {}^{x^{-1}h}R\le Q}} [N_P(R)/N_{{}^{h^{-1}x}Q}(R)]\\
&=& \sum_{\substack{g\in Q\backslash P/N_P(R):\\ N\le {}^gR\le Q}} [N_P(R)/N_{{}^{g^{-1}}Q}(R)] \\ &=& a_R(Q,N)
\end{eqnarray*}
as required. Thus the map $\alpha_P$ is well-defined.

By definition $\alpha_P$ is $\mathbb Q$-linear. To prove that it is bijective, we first show that the domain and the range of $\alpha_P$ 
have the same $\mathbb Q$-dimension, and then show that the map $\alpha_P$ is injective. Recall that the set $\mathcal{SQ}_P(P)$ 
of all 
$P$-conjugacy classes of all pairs $(Q,N)$ with $N\unlhd Q\le P$ is a basis for the vector space $\mathbb Q\Lambda (P)$. On the other hand, a basis for the vector space
\[
\mathbb QB^\oplus (P) := \bigoplus_{R\in s(P)} \mathbb QB(N_P(R)/R)
\]
is the set
\[
\mathcal{SS}_P(P) := \bigcup_{R\in s(P)} \mathcal S(N_P(R)/R)
\]
where $\mathcal S_G(G)$ denotes the set of conjugacy classes of subgroups of $G$ for any finite group $G$. For simplicity, we write 
$[T]_R$ for the $N_P(R)/R$-conjugacy class of $T/R$. Therefore our claim is equivalent to prove
that $|\mathcal{SQ}_P(P)| = |\mathcal{SS}_P(P)|$. We define two functions
\begin{align*}
  \tilde f \colon \mathcal{SS}_P(P)&\to \mathcal{SQ}_P(P)\\
   [T]_R&\mapsto [T,R]_P.
\end{align*}
and
\begin{align*}
  f \colon \mathcal{SQ}_P(P)&\to \mathcal{SS}_P(P)\\
   [Q,N]_P&\mapsto [\{Q\}]_{\{N\}}
\end{align*}
where $\{N\}$ denotes the representative of the $P$-conjugacy class of $N$ in the set $s(P)$ and putting $\{N\} = {}^xN$, we
define $\{Q\}$ as the $N_P(\{N\})/\{N\}$-conjugacy class of ${}^xQ/{}^xN$. We have to show that $f$ and $\tilde f$ are well-defined
mutually inverse functions. 

The function $\tilde f$ is well-defined: if $[T]_R = [T^\prime]_{R^\prime}$, then $R = R^\prime$, and if
$T^\prime = {}^xT$ for some $x\in N_P(R)$, we get
\[
\tilde f([T^\prime]_R) = [T^\prime, R]_P = [{}^xT, R]_P = [{}^xT, {}^xR]_P = [T, R]_P = \tilde f([T]_R). 
\]
To see that the function $f$ is well-defined, let $(L,M) = {}^g(Q,N)$ for some $g\in P$. In particular, the subgroups $M$ and $N$
are $P$-conjugate and hence $\{M\} = \{N\}$. Let ${}^xN = \{M\} = {}^yM$. Then
\[
[{}^{y^{-1}x}L] = [{}^{y^{-1}xg}Q] = [Q]
\]
as $N_P(M)$-conjugacy classes since by the above choices, we have $y^{-1}xg\in N_P(M)$. Thus, $[{}^xL] = [{}^xQ]$ as
$N_P(\{M\})$-conjugacy classes and thus the function $f$ is well-defined.

Next, we show that the functions $f$ and $\tilde f$ are inverse functions. We have
\[
(\tilde f\circ f) ([Q,N]_P) = \tilde f([\{Q\}]_{\{N\}}) = [\{Q\}, \{N\}]_P. 
\]
Since $(\{Q\}, \{N\})$ is $P$-conjugate to $(Q,N)$, by definition, we get $\tilde f\circ f = \mbox{\rm id}_{\mathcal{SQ}_P(P)}$. 
Similarly, 
\[
(f\circ \tilde f) ([T]_R) = f([T,R]_P) = [\{T\}]_{\{R\}}. 
\]
Since $R\in s(P)$, we have $\{R\} = R$, and hence $\{T\}$ is $N_P(R)$-conjugate to $T$. Therefore $[\{T\}]_{\{R\}} = [T]_R$
and hence $f\circ \tilde f = \mbox{\rm id}_{\mathcal{SS}_P(P)}$. 

Finally, to complete the proof, we show that the map $\alpha_P$ is injective. Let
\[
u = \sum_{Q/N\squ_P P} u_{Q,N} [Q,N]_P 
\]
such that not all $u_{Q,N}$ are zero but $\alpha_P(u) = 0$. Let $R$ be a subgroup of $P$ of minimal order such that $u_{Q,R} 
\neq 0$ for some $Q\le P$ and, without loss, suppose that $R\in s(P)$. Then, since $\alpha_P(u) = 0$,
\begin{eqnarray*}
0 = a_R(u) &=& \sum_{Q/N\squ_P P} u_{Q,N} a_R(Q,N)\\
 &=&  \sum_{Q/N\squ_P P} u_{Q,N} \sum_{\substack{g\in Q\backslash P/N_P(R)\\N\le {}^gR\le Q}} [N_P(R)/N_{{}^{g^{-1}}Q}(R)].
\end{eqnarray*}
By the choice of $R$, the only non-zero contribution to the sum comes from the pairs where $N = {}^gR$ for some $g\in P$. Since
$N$ and $R$ are both taken up to $P$-conjugacy, there is only one possible choice for $N$ and since
$a_R(Q,N)$ is independent of the choice of the representative $(Q,N)$ of the class $[Q,N]_P$, we may let $N = R$. In this
case, the condition $N = {}^gR$ implies that $g\in N_P(R)$, and hence there is only the trivial $(Q, N_P(R)$-double coset to consider
in the second sum of the last sum above. Moreover, the indecesof the former sum run through the conjugacy classes subgroups of $N_P(R)$ containing $R$. Hence, we obtain
\begin{eqnarray*}
0 = a_R(u) &=&  \sum_{R\le Q\le_{N_P(R)}N_P(R)} u_{Q,R} [N_P(R)/Q].
\end{eqnarray*}
Now, since the set $\{[N_P(R)/Q] | R\le Q\le_{N_P(R)} N_P(R)\}$ is a basis for the $\mathbb Q$-vector space 
$\mathbb QB(N_P(R)/R)$, the above equality implies that $u_{Q,R} = 0$ for each $Q$. But this contradicts to the choice of
$R$. Hence there is no such $u$ and $\alpha_P$ is injective.
\end{proof}

The above isomorphism and the isomorphism of Theorem \ref{thm:tw18} are related by the following commutative diagram.
\begin{cor}
Assume the above notation, and let $\lin^\oplus_P$ denote the sum of the linearization maps $\lin_{N_P(R)/R}$ over $R\in s(P)$.
Then the diagram
\begin{equation*}
  \xymatrix@C+2em@R+2em{
   \mathbb Q\Lambda(P) \ar[d]_{\alpha_P} \ar[r]^{\lin_P^\mu} & \mathbb Q\mathcal R_{\mu_\mathbb Q}(P) \ar[d]^{\overline{?}}\\
   \bigoplus_{R\in s(P)}\mathbb Q B(N_P(R)/R) \ar[r]_{\lin^\oplus_P} &  \bigoplus_{R\in s(P)}\mathbb Q\mathcal R_\mathbb K
(N_P(R)/R)
  }
 \end{equation*}
commutes.
\end{cor}
\begin{proof}
Since all of the maps shown in the diagram are linear, it is sufficient to prove the claim for a 
basis element. Hence let $R/N$ be a subquotient of $P$. We need to show that
\[
\overline{\lin_P^\mu([R/N]_P)} = \lin_P^\oplus(\alpha_P([R/N]_P)).
\]
Note that, by definition, we have
\[
\lin_P^\mu([R/N]_P) = \ind_R^P\infl_{R/N}^R B^{R/N}.
\]
However by \cite[Section 6]{Clin}, we have
\[
\lin_P^\mu ([R/N]_P) = \bigoplus_{\substack{N\le K\le_R R}} Y_{K, \ind_{W_R(K)}^{W_P(K)}1}.
\]
Here we write $W_R(K) = N_R(K)/K$ to simplify the notation.
Also, since for any $W_P(K)$-module $m$, we have $\overline{Y_{K,m}}(Q) = 0$ unless $K=_P Q$ and $\overline{Y_{K,m}}(K) = m$, 
it follows that the $Q$-th component of $\overline{\lin_P^\mu([R/N]_P)}$ is
\begin{equation}\label{eqn:1}
(\overline{\lin_P^\mu([R/N]_P)})_Q \cong \bigoplus_{\substack{N\le K\le_R R\\ K=_P Q}} \ind^{W_P(Q)}_{W_{{}^{g^{-1}}R}(Q)}1
\end{equation}
where, for each $P$-conjugate $K$ of $Q$, the element $g$ is chosen so that $K = {}^gQ$.
On the other hand, by the definition of $\alpha_P$, the $Q$-th component of $\alpha_P([R/N]_P)$ is
\[
a_Q([R/N]_P) = \sum_{g\in R\backslash P/N_P(Q) : N\le {}^gQ\le R}  [N_P(Q)/N_{{}^{g^{-1}}R}(Q)] 
\]
Thus the $Q$-th component of $\lin_P^\oplus(\alpha_P([R/N]_P))$ is 
\begin{equation}\label{eqn:2}
(\lin_P^\oplus(\alpha_P([R/N]_P)))_Q \cong \bigoplus_{\substack{g\in R\backslash P/N_P(Q) \\ N\le {}^gQ\le R}}  
\ind^{N_P(Q)}_{N_{{}^{g^{-1}}R}(Q)}1.
\end{equation}
Finally, since the index set of Equation (\ref{eqn:1}) and Equation (\ref{eqn:2}) are the same, equality holds and the diagram 
commutes.
\end{proof}

As another corollary of the above theorem, we get an isomorphism 
\[
\ker(\lin_P^\mu) \cong \ker(\lin_P^\oplus).
\]
Now by \cite[Theorem D]{BT}, for any finite $p$-group $P$, we have
 \[
 \mathbb QD(P) \cong \ker(\lin_P:\mathbb QB(P)\rightarrow \mathbb Q\mathcal R_\mathbb Q(P)).
 \]
 Therefore we have
 \[
 \ker(\lin_P^\oplus) \cong \bigoplus_{R\in s(P)} \mathbb QD(N_P(R)/R).
 \]
 Thus, by Theorem \ref{thm:dadeiso}, we get an isomorphism
\[
\ker(\lin_P^\mu) \cong \mathbb QD_\mu(P)
\]  
together with the map 
$$\iota_P^\mu:=\alpha_P^{-1}\circ \iota^\oplus_P\circ \bar ?: \mathbb QD_\mu(P)\rightarrow \mathbb Q\Lambda(P) $$
 where we write $\iota_P: \mathbb QD(P)\rightarrow \mathbb QB(P)$ for the map $\alpha$ of \cite[Theorem D]{BT} and 
 $\iota_P^\oplus$ for
 the corresponding direct sum of $\iota_{N_P(R)/R}$'s. Now since $\alpha_P$ and
 $\bar ?$ are isomorphisms and $\iota_P^\oplus$ is injective, the map $\iota_P^\mu$ is also injective. Moreover, by the
 commutativity of the above diagram, for any endo-permutation Mackey functor $M$, we have
 \begin{eqnarray*}
(\lin_P^\mu\circ\iota_P^\mu)(M) =  \lin_P^\mu((\alpha_P^{-1}\circ \iota^\oplus_P) (\overline{M})) &=& (\varrho_P\circ \lin_P^\oplus\circ \iota_P^\oplus)( 
 \overline{M})= 0 
 \end{eqnarray*}
 Therefore, we obtain the following result.
\begin{thm}
Let $P$ be a $p$-group. The sequence
\begin{equation*}
  \xymatrix@C+2em@R+2em{
   0\ar[r]& \mathbb QD_\mu(P) \ar[r]^{\iota^\mu_P}& \mathbb Q\Lambda(P) \ar[r]^{\lin_P^\mu} & \mathbb Q\mathcal R_{\mu_\mathbb Q}(P) \ar[r] & 0
  }
 \end{equation*}
of $\mathbb Q$-vector spaces is exact.
 \end{thm}
 
 We finish by a result to show that the above analogy breaks for the following next step. By \cite[Statement 2.1.2]{P}, there is an isomorphism
 \[
 T(P) \cong \bigcap_{1<Q\le P}\mbox{\rm Ker}(\defl^{N_P(Q)}_{N_P(Q)/Q}\res^P_{N_P(Q)} :D(P)\to D(N_P(Q)/Q))
 \]
where we write $T(P)$ for the group of endo-trivial $kP$-modules. Our next result shows that the above analogy between 
endo-permutation Mackey functors and endo-permutation modules does not hold in this direction. Indeed, since we have defined
restriction and deflation maps for the Mackey-Dade groups, the version of the right-hand side can be evaluated for the 
Mackey-Dade groups. However, as we shall see, it is not isomorphic to the direct sum of the groups of endo-trivial 
$kN_P(Q)/Q$-modules as $Q$ runs over all subgroup of $P$ up to conjugation. In a future paper, we are aiming to study 
endotrivial Mackey functors and their relations to endo-permutation Mackey functors.
 
To describe the intersection of kernels of deflation-restriction maps for the Mackey-Dade group, we need the 
following definition of the endo-permutation module $\Delta(P)$ from \cite{BTen}. Let $M(P)$
denote the disjoint union of the $P$-sets $P/Q$ as $Q$ runs over all maximal subgroups of the finite $p$-group $P$ and let
\[
\Delta(P) = \Omega_{M(P)}
\] 
denote the corresponding relative syzygy, as defined by Bouc in \cite[Notation 6.2.1]{BTen}. Then by \cite[Remark 6.2.2]{BTen},
the class of $\res^P_R \Delta(P)$ in the Dade group $\mathbb QD(R)$ is zero when $R\neq P$ and moreover $\Delta(P)$ is $\out(P)$-invariant.
Finally, by \cite[Remark 7.4.11]{BTen}, the elements $\ten_R^P \Delta(R)$ form a basis of $\mathbb QD(P)$, as $R$ runs over
all non-cyclic subgroups of $P$. 
 
 \begin{pro}
Let $P$ be a non-cyclic $p$-group. Then
\[
\underline{\mathbb Q D_\mu} (P) := \bigcap_{R/N\squ P} \ker\big(\jef^R_{R/N}\res^P_R: \mathbb QD_\mu(P)\rightarrow 
\mathbb Q D_\mu(R/N)\big)
\]
is equal to the $\mathbb Q \out(P)$-submodule of $\mathbb QD_\mu(P)$ generated by $\varrho_P(\Delta(P))$.
\end{pro}

\begin{proof}
Let $M$ be an endo-permutation Mackey functor for $P$ and let $1\neq Q \le P$. Then by Theorem \ref{pro:barres}, we have
\[
\overline{\jef^{N_P(Q)}_{N_P(Q)/Q}\res^P_{N_P(Q)} M} (Q/Q) \cong \overline M (Q)
\]
as $kN_P(Q)/Q$-modules. This implies, in particular, that if $M$ is in $\underline{D_\mu}(P)$, then $\overline M(Q)$ is zero in 
$\mathbb QD(N_P(Q)/Q)$. In other words, $[M]\in \underline{\mathbb Q D_\mu}(P)$ only if $[\overline M (Q)] = 0$
for any $1\neq Q\le P$. 

Regarding the endo-permutation $kP$-module $M(1)$, we have
\[
(\res^P_R M)(1) = \res^P_R (M(1)).
\]
Thus $[M]\in \underline{\mathbb QD_\mu}(P)$ implies that $\res^P_R [M(1)] = 0$ in $\mathbb QD(R)$ for any $R< P$.
Since the proper subgroups $R$ of $P$ are the only subquotients of $P$ for which the action of an $(R,P)$-biset contribute to 
 $M(1)$, we have that $[M]\in \underline{\mathbb QD_\mu}(P)$ if and only if
\begin{enumerate}
\item[(i)] $[\overline M(Q)]= 0$ in $\mathbb QD(N_P(Q)/Q)$ for each non-trivial subgroup $1\neq Q\le P$ and
\item[(ii)] $\res^P_R[\overline M(1)] = 0$ in $\mathbb QD(R)$ for each proper subgroup $R< P$.
\end{enumerate}
Therefore, to determine the subgroup $\underline{\mathbb QD_\mu}(P)$, we only need to calculate the subgroup
\[
\underline{ \mathbb Q D}(P) = \bigcap_{R< P} \ker\big( \res^P_R : \mathbb Q D(P) \rightarrow \mathbb Q D(R) \big).
\]
We claim that this subgroup is generated by the endo-permutation $kP$-module $\Delta(P)$. Indeed, let $x\in 
\underline{\mathbb Q D}(P)$. Since the elements $\ten_R^P \Delta(R)$ form a basis of $\mathbb Q D(P)$, we can write
\[
x = \sum_{R\le_P P} a_R\ten_R^P \Delta(R)
\]
for some $a_R\in \mathbb Q$. We want to show that all the coefficients $a_Q$ are zero except possibly for $a_P$. To prove this
claim, note that, for a fixed subgroup $R\le P$, and a subgroup $Q< P$, the Mackey formula implies
\begin{eqnarray*}
\res^P_Q\ten_R^P \Delta(R) &=& \sum_{x\in Q\backslash P/R}\ten^{Q}_{Q\cap {}^xR}\res^{{}^xR}_{Q\cap {}^xR} \Delta({}^xR)\\
&=& \sum_{\substack{x\in Q\backslash P/R\\ {}^xR\le Q}}\ten^{Q}_{{}^xR} \Delta({}^xR).
\end{eqnarray*}
Here we also use that the restriction of $\Delta(R)$ to any proper subgroup of $R$ is zero. In particular, if the order of $Q$ is less
than or equal to the order of $R$ (and if $Q$ is not $P$-conjugate to $R$, when they have equal orders), then the condition 
${}^xR\le Q$ is not satisfied for any $x\in P$ and we get
\[
\res^P_Q\ten^P_R \Delta(R) = 0.
\]
Now, let $Q\le P$ be a subgroup of minimal order such that the coefficient $a_Q\neq 0$ and suppose that $Q\neq P$. Then by
the above calculation and by the choice of $x$, we have
\begin{eqnarray*}
0 = \res^P_Q x &=& \sum_{R\le_P P}a_R\res^P_Q\ten_R^P \Delta(R)\\
&=& a_Q \res^P_Q\ten_Q^P \Delta(Q)\\
&=& a_Q |N_P(Q):Q| \Delta(Q).
\end{eqnarray*}
Therefore, we get that $a_Q =0$, a contradiction. Hence we have $Q=P$, and $x = a_P\Delta(P)$ for some $a_P\in \mathbb Q$.   

Finally, with the above identification of $\underline{\mathbb QD_\mu}(P)$, we get that
\[
\underline{\mathbb QD_\mu}(P) \cong \langle\varrho_P(\Delta(P))\rangle
\]
where $\varrho_P$ is as defined in Section \ref{sec:dd}. This isomorphism is an isomorphism of $\mathbb Q\out(P)$-modules
since $\Delta(P)$ is $\out(P)$-invariant.
\end{proof}

\section*{Acknowledgement}
I would like to thank to the referee for his/her valuable comments ard corrections. 

The author is supported by Boğaziçi University Research Fund Grant Number 15B06P1.

\end{document}